\documentclass[12pt]{amsart}

\usepackage{latexsym, amssymb, amsmath,ulem,soul,esint}

\usepackage{amsfonts, graphicx}
\usepackage{graphicx,color}
\usepackage{url}

\newcommand{\be}{\begin{equation}}
\newcommand{\ee}{\end{equation}}
\newcommand{\beq}{\begin{eqnarray}}
\newcommand{\eeq}{\end{eqnarray}}

\newtheorem{thm}{Theorem}[section]

\newtheorem{lma}{Lemma}[section]
\newtheorem{prop}{Proposition}[section]
\newtheorem{cor}{Corollary}[section]
\newtheorem{defn}{Definition}[section]

\theoremstyle{remark}

\numberwithin{equation}{section}

\def\be{\begin{equation}}
\def\ee{\end{equation}}
\def\bee{\begin{equation*}}
\def\eee{\end{equation*}}

\newcommand{\Ric}{\mathrm{Ric}}

\def\Ric{\text{\rm Ric}}

\def\p{\partial}

\def\e{\varepsilon}
\def\a{{\alpha}}

\begin{document}

\title{Rigidity on non-negative intermediate curvature}

\author[J. Chu]{Jianchun Chu$^1$}
\address[Jianchun Chu]{School of Mathematical Sciences, Peking University, Yiheyuan Road 5, Beijing, P.R.China, 100871}
\email{jianchunchu@math.pku.edu.cn}

\author[K.-K. Kwong]{Kwok-Kun Kwong$^2$}
\address[Kwok-Kun Kwong]{School of Mathematics and Applied Statistics, University of Wollongong, Northfields Ave, NSW 2522, Australia}
\email{kwongk@uow.edu.au}

\author[M.-C. Lee]{Man-Chun Lee$^3$}
\address[Man-Chun Lee]{Department of Mathematics, The Chinese University of Hong Kong, Shatin, N.T., Hong Kong}
\email{mclee@math.cuhk.edu.hk}

\thanks{$^1$Research partially supported by Fundamental Research Funds for the Central Universities (No. 7100603624).}

\thanks{$^2$Research partially supported by the CERL fellowship at University of Wollongong.}

\thanks{$^3$Research partially supported by Hong Kong RGC grant (Early Career Scheme) of Hong Kong No. 24304222 and a direct grant of CUHK}

\date{\today}

\begin{abstract}
In a recent work of Brendle-Hirsch-Johne, a notion of intermediate curvature was introduced to extend the classical non-existence theorem of positive scalar curvature on torus to product manifolds. In this work, we study the rigidity when the intermediate curvature is only non-negative when the ambient dimension is at most $5$.
\end{abstract}


\maketitle

\section{Introduction}

In differential geometry, there has been great interest in studying the connection between the curvature and properties of the underlying manifold. For example, a conjecture of Geroch \cite{Geroch} (cf. also \cite{KW}) asks if there is a complete metric with positive scalar curvature on the torus $\mathbb T^n$. Indeed, the original formulation concerns a special case of the positive mass theorem (unsolved at that time): the Euclidean metric on $\mathbb R^n$ does not admit a non-trivial compactly supported perturbation while keeping the scalar curvature non-negative. The conjecture was solved by Schoen and Yau \cite{SchoenYau1979} for $3\le n\le 7$ by using minimal hypersurfaces, and by Gromov and Lawson \cite{GromovLawson} for all $n$ by using spinors. Relatively recently, Schoen and Yau \cite{SchoenYau2021} generalized the minimal hypersurface argument to all dimensions and prove the positive mass theorem by constructing minimal slicings.

More recently, Brendle, Hirsch and Johne \cite{BrendleHirschJohne2022} developed the method of stabled weighted slicings, which closely resembles the notion of minimal-slicings by Schoen and Yau \cite{SchoenYau1979, SchoenYau2021}, to study a generalization of the Geroch's conjecture. They defined a new notion of curvature, called the $m$-intermediate curvature $\mathcal{C}_m$, which interpolates between the Ricci curvature and the scalar curvature (and some other curvatures such as the bi-Ricci and tri-Ricci curvature) by choosing different values of $m$. For the reader's convenience, let us recall the definition of $\mathcal C_m$:

\begin{defn}
Suppose $(N^n,g)$ is a Riemannian manifold and $1\leq m\leq n-1$. We say that its $m$-intermediate curvature is non-negative if for any $x\in M$,
\begin{equation}
\mathcal{C}_m(e_1,...,e_m) =\sum_{p=1}^m \sum_{q=p+1}^n \mathrm{Rm}(e_p,e_q,e_p,e_q)\geq 0
\end{equation}
for any orthonormal basis $\{e_i\}_{i=1}^n$ of $T_xN$.  If the inequality holds for all  orthonormal frame, we will denote it using $\mathcal{C}_m\geq 0$ for notational convenience. The positivity and quasi-positivity of $\mathcal{C}_m$ can be defined analogously.
\end{defn}

When $m=1$, $\mathcal{C}_1$ refers to the Ricci curvature while when $m=n-1$, $\mathcal{C}_{n-1}$ corresponds to the scalar curvature (up to a multiplicative constant). The product manifold $N^{n}=M^{n-m} \times \mathbb{T}^{m}$ has  $\mathcal C_{m+1}>0$ and $\mathcal C_{m}\ge0$, if $M$ is positively curved.

Using the method of stable weighted slicings, Brendle, Hirsch and Johne \cite{BrendleHirschJohne2022} proved that, among other things, the product manifold $N^n=M^{n-m}\times \mathbb T^m$ does not admit a metric of positive $m$-intermediate curvature
for $n \le 7$ and for any compact manifold $M^{n-m}$. This theorem reduces to the solution to the  classical Geroch's conjecture when $m=n-1$. They also proved that under some condition on $m$ and $n$, if $\left(N^{n}, g\right)$ is a compact Riemannian manifold with positive $m$-intermediate curvature, then $N$ does not admit a stable weighted slicing of order $m \le n-1$.

The study of the topological or analytic properties of complete manifolds with non-negative Ricci curvature or scalar curvature has of course attracted great interest and is still being studied intensively, especially in higher dimensions. See for example \cite{SchoenYau1982}. We  refer readers to Gromov's  lecture notes \cite{GromovLecture} for a comprehensive overview on the recent development on scalar curvature geometry. In view of the above mentioned results in \cite{BrendleHirschJohne2022}, it is natural to ask whether we can obtain some rigidity results for manifolds which admit a metric with non-negative $m$-intermediate curvature, under some topological conditions. When $m=n-1$, it is well-known that if a compact manifold does not admit a metric with positive scalar curvature, then any metric with non-negative scalar curvature must be Ricci flat, using a variational argument. Hence any metric with non-negative scalar curvature on a torus must be flat by the splitting Theorem.  This is closely related to the rigidity in positive mass theorem. It is therefore natural to ask if we can extend the rigidity in case of $m=n-1$ to general $m\leq n-1$. 
In this regard, we show a splitting Theorem, which roughly says that a compact manifold $N^n$ with $\mathcal C_m\ge 0$ has a cover which splits off a factor of $\mathbb R^{m}$, if there is a non-zero degree map from $N^n$ to $M^{n-m} \times \mathbb{T}^{m} $ for some compact manifold $M^{n-m}$.

\begin{thm}\label{main-rigidity}
Suppose $N^n$ is a compact manifold such that there exists a non-zero degree map $f:N^n\to M^{n-m}\times \mathbb{T}^m$ for some compact manifold $M$ and $1\leq m\leq n-1$ and $n\leq 5$. If $N$ admits a metric $g$ with $\mathcal{C}_{m}\geq 0$, then $(N^n,g)$ is isometrically covered by $(X^{n-m}\times \mathbb{R}^m,g_{X}+g_{\mathbb{R}^m})$ for some compact manifold $X$ with $\Ric(g_{X})\geq 0$ and hence $\Ric(g)\geq0$. Furthermore, if $\mathrm{scal}(N)\equiv 0$, then $\Ric(g_X)\equiv 0$ and $\Ric(g)\equiv 0$.
\end{thm}

When $m=n-1$, it reduces back to the classical rigidity on torus. As a direct corollary, we generalize \cite[Corollary 1.6]{BrendleHirschJohne2022} to the quasi-positive case when $n\leq 5$.
\begin{cor}\label{main-rigidity-quasi}
Suppose $N^n$ is a compact manifold such that there exists a non-zero degree map $f:N^n\to M^{n-m}\times \mathbb{T}^m$ for some compact manifold $M$ with $n\leq 5$, then $N$ does not admit metric such that $\mathcal{C}_m\geq 0$ on $N$ and $\mathcal{C}_m>0$ at some point $x\in N$.
\end{cor}

It is also interesting to compare with the Ricci case i.e. $m=1$. For example, it is well-known that $\mathbb{S}^p\times \mathbb{S}^{1}$ cannot admit metric with $\mathcal{C}_1\equiv 0$ by the splitting theorem for $p=2$, $3$. Motivated by this, we combine Theorem~\ref{main-rigidity} with the celebrated works of Hamilton \cite{Hamilton3,Hamilton4} to give a classification of product manifolds with $\mathcal{C}_m\geq 0$ when $n-m\leq 3$.

\begin{cor}\label{main-topo-lowD}
Under the assumption in Theorem~\ref{main-rigidity}, if in addition $n-m\leq 3$, then $N$ is covered by $\mathbb{R}^n$, $\mathbb{S}^2\times \mathbb{R}^{n-2}$ or $\mathbb{S}^3\times \mathbb{R}^{n-3}$. Furthermore,
\begin{enumerate}\setlength{\itemsep}{1mm}
    \item [(i)] if $\mathrm{scal}(N)=0$, then $N$ is isometrically covered by $\mathbb{R}^n$;
    \item[(ii)] if $\mathrm{scal}(N)$ is positive at one point and $n-m=2$, then $N$ is isometrically covered by $(\mathbb{S}^2, g_S)\times \mathbb{R}^{n-2}$ for some metric $g_S$ on $\mathbb S^2$.
\end{enumerate}
\end{cor}
In particular, this says that $\mathbb{S}^{n-m}\times \mathbb{T}^m$ cannot admit metrics with $\mathcal{C}_m\equiv 0$ if $n-m=2$, $3$ and $n\leq 5$. 

The proof of Theorem~\ref{main-rigidity} relies on a variational argument of Zhu \cite{Zhu2020} in producing minimal foliation, see also the work of Bray, Brendle and Neves \cite{BrayBrendleNeves2010}.  The  restriction on the dimension in this work is in spirit similar to that in  \cite[Theorem 1.4]{BrendleHirschJohne2022} which is not due to the regularity issue in geometric measure theory. In view of Theorem 1.5 in the work \cite{BrendleHirschJohne2022} of Brendle, Hirsch and Johne, we expect that the rigidity still holds for all $n\leq 7$.

The organization of this paper is as follows. In Section 2, we gather some knowledge about stable weighed slicings, which is essentially all contained in \cite{BrendleHirschJohne2022}. It will be needed in the proof of our main result. In Section 3, we prove a foliation result under the assumption that $\mathcal C_m\ge 0$. In Section 4 we prove our main results.

\section{Results of Brendle-Hirsch-Johne}

In this section, we will outline some important results obtained in \cite{BrendleHirschJohne2022}. The precise estimates will be needed in the rigidity. We start with the concept of stable weighted slicing which was introduced in \cite{BrendleHirschJohne2022}.

\begin{defn}\label{defn:Stable weighted slicing of order m} (Definition 1.3 in \cite{BrendleHirschJohne2022})
Suppose $(N^n,g)$ be a Riemannian manifold and $1\leq m\leq n-1$. A stable weighted slicing of order $m$ consists of submanifolds $\Sigma_k$ and positive functions $\rho_k\in C^\infty(\Sigma_k)$, $0\leq k\leq m$, satisfying the following conditions:
\begin{enumerate}\setlength{\itemsep}{1mm}
\item [(a)]$\Sigma_0=N$ and $\rho_0=1$;
\item [(b)] For each $1\leq k\leq m$, $\Sigma_k$ is an embedded two-sided hypersurface in $\Sigma_{k-1}$, and is also a stable critical point of the $\rho_{k-1}$-weighted area:
$$\mathcal{H}^{n-k}_{\rho_{k-1}}(\Sigma)=\int_{\Sigma} \rho_{k-1} \, d\mu.$$
\item[(c)] For each $1\leq k\leq m$, the function $\frac{\rho_k}{\rho_{k-1}|_{\Sigma_k}}\in C^\infty(\Sigma_k)$ is a first eigenfunction of the stability operator associated with the $\rho_{k-1}$-weighted area with eigenvalue $\lambda_k\geq 0$.
\end{enumerate}
\end{defn}

Throughout this section, we are going to assume that a stable weighted slicing $\Sigma_m\subset \cdots\subset \Sigma_0=N$ exists and is fixed.
Following \cite{BrendleHirschJohne2022}, we will write $w_k=\log \frac{\rho_k}{\rho_{k-1}|_{\Sigma_k}}$ and denote the unit normal, second fundamental form and mean curvature of $\Sigma_{k}$ in $\Sigma_{k-1}$ by $\nu_{k}$, $h_{\Sigma_{k}}$ and $H_{\Sigma_{k}}$. Then $\Ric_{\Sigma_{k-1}}(\nu_k,\nu_k)$ is the normal Ricci curvature of $\Sigma_k$ in $\Sigma_{k-1}$. Moreover, we will use $ \mathrm{Rm}(X,Y,X,Y)$ to denote the sectional curvature (up to scaling) of the plane spanned by $X,Y\in TN$.

\begin{lma}[Lemma 3.4 in \cite{BrendleHirschJohne2022}]\label{lma:Stable-ineq}
For $m\geq 2$, if we denote
\begin{equation}
\left\{
\begin{array}{ll}
\displaystyle \Lambda=\sum_{k=1}^{m-1} \lambda_k;\\[5mm]
\displaystyle \mathcal{R}=\sum_{k=1}^m \Ric_{\Sigma_{k-1}}(\nu_k,\nu_k);\\[5mm]
\displaystyle \mathcal{G}=\sum_{k=1}^{m-1} \langle \nabla_{\Sigma_k}\log\rho_k,\nabla_{\Sigma_k}w_k \rangle;\\[5mm]
\displaystyle \mathcal{E}= \sum_{k=1}^m |h_{\Sigma_k}|^2 -\sum_{k=2}^m H_{\Sigma_k}^2,
\end{array}
\right.
\end{equation}
then
\begin{equation}\label{Stable-ineq 1}
\int_{\Sigma_m} \rho_{m-1}^{-1} \left(\Lambda+ \mathcal{R}+\mathcal{G}+\mathcal{E} \right)\; d\mu \leq 0.
\end{equation}
\end{lma}

When $m=1$, $\mathcal{C}_1\geq 0$ is equivalent to $\Ric_{\Sigma_0}=\Ric_{N}\geq0$. By taking a constant test function in the stability inequality of $\Sigma_{1}$, we immediately see that
\begin{equation}\label{m = 1 case}
\Ric_{\Sigma_0}(\nu_{1},\nu_{1}) = 0, \ \ h_{\Sigma_1} = 0  \ \ \text{on $\Sigma_{1}$}.
\end{equation}
When $\mathcal{C}_m\geq 0$ for $m\geq2$, it was shown that if in addition $n(m-2)\leq m^2-2$, then the sum in \eqref{Stable-ineq 1} can be expressed as a sum of intermediate curvature and a non-negative term.
\begin{lma}[Lemma 3.10, 3.11, 3.12 and 3.13 of \cite{BrendleHirschJohne2022}]\label{lma:BHJ-ineq}
Suppose $ n(m-2)\leq m^{2}-2$ and $m\geq 2$. For $x\in\Sigma_{m}$, let $\{e_{i}\}_{i=1}^{n}$ be an orthonormal basis of $T_{x}N$ such that $e_{k}=\nu_{k}$ for $1\leq k\leq m$. Then the following holds:
\begin{equation}
\mathcal{R}+\mathcal{G}+\mathcal{E} \geq \mathcal{C}_{m}(e_{1},\ldots,e_{m})+\sum_{k=1}^{m}\mathcal{V}_{k},
\end{equation}
where
\begin{equation*}
\begin{split}
\mathcal{V}_{1} & = |h_{\Sigma_{1}}|^{2}+\sum_{p=2}^{m}\sum_{q=p+1}^{n}\left(h_{\Sigma_{1}}(e_{p},e_{p})h_{\Sigma_{1}}(e_{q},e_{q})-h_{\Sigma_{1}}(e_{p},e_{q})^{2}\right), \\
\mathcal{V}_{k} & = |h_{\Sigma_{k}}|^{2}-\left(\frac{1}{2}-\frac{1}{2(k-1)}\right)H_{\Sigma_{k}}^{2} \\
& + \sum_{p=k+1}^{m}\sum_{q=p+1}^{n}\left(h_{\Sigma_{k}}(e_{p},e_{p})h_{\Sigma_{k}}(e_{q},e_{q})-h_{\Sigma_{k}}(e_{p},e_{q})^{2}\right) \ \ \text{for $2\leq k\leq m-1$}, \\[1mm]
\mathcal{V}_{m} & = |h_{\Sigma_{m}}|^{2}-\left(\frac{1}{2}-\frac{1}{2(m-1)}\right)H_{\Sigma_{m}}^{2}.
\end{split}
\end{equation*}
Moreover, for each term $\mathcal{V}_{k}$, the following estimates hold:
\begin{equation}
\begin{split}
\mathcal{V}_{1} & \geq \frac{m^{2}-2-n(m-2)}{2(m-1)(n-m)}\left(\sum_{p=2}^{m}h_{\Sigma_{1}}(e_{p},e_{p})\right)^{2}, \\
\mathcal{V}_{k} & \geq \frac{m^{2}-2-n(m-2)}{2(m-1)(n-m)}\left(\sum_{q=m+1}^{n}h_{\Sigma_{k}}(e_{q},e_{q})\right)^{2} \ \ \text{for $2\leq k\leq m-1$}, \\[2mm]
\mathcal{V}_{m} & \geq \frac{m^{2}-2-n(m-2)}{2(m-1)(n-m)}H_{\Sigma_{m}}^{2}.
\end{split}
\end{equation}
\end{lma}

In particular,  we have the following equality when $\mathcal{C}_m$ is only non-negative.
\begin{lma}\label{V H C vanish}
Suppose $n(m-2)<m^{2}-2$ and $m\geq2$. For $x\in\Sigma_{m}$, let $\{e_{i}\}_{i=1}^{n}$ be an orthonormal basis of $T_{x}N$ such that $e_{k}=\nu_{k}$ for $1\leq k\leq m$. If $\mathcal{C}_{m}(e_{1},\ldots,e_{m})\geq 0$, then the followings hold:
\begin{equation}
\begin{split}
\mathcal{V}_{1} = {} & 0, \quad   \sum_{p=2}^{m}h_{\Sigma_{1}}(e_{p},e_{p}) = 0, \\
\mathcal{V}_{k} = {} & 0, \ \   \sum_{q=m+1}^{n}h_{\Sigma_{k}}(e_{q},e_{q}) = 0 \ \ \text{for $2\leq k\leq m-1$}, \\[3.5mm]
\mathcal{V}_{m} = {} & 0, \quad   H_{\Sigma_{m}} = 0, \quad \mathcal{C}_{m}(e_{1},\ldots,e_{m}) = 0.
\end{split}
\end{equation}
\end{lma}

We now observe that the estimate in the rigid case can be made better by squeezing the equality case.
\begin{prop}\label{Prop:rigidity}
Under the assumption in Lemma~\ref{V H C vanish}, then the following is true:
\begin{enumerate}\setlength{\itemsep}{1mm}
\item [(i)] $\lambda_k=0$ for all $1\leq k\leq m-1$;
\item[(ii)] $\nabla_{\Sigma_{m-1}}\log\rho_{m-1}=0$ along $\Sigma_m$;
\item[(iii)]$\Ric_{\Sigma_{m-1}}(\nu_m,\nu_m)=\left(\nabla^2_{\Sigma_{m-1}}\log\rho_{m-1}\right) (\nu_m,\nu_m)$ on $\Sigma_m$;
\item[(iv)] $h_{\Sigma_k}=0$ for all $1\leq k\leq m$ on $\Sigma_m$;
\item[(v)] $ \mathcal{C}_{m}(e_{1},\ldots,e_{m}) = 0$;
\item[(vi)] $\mathrm{scal}(g|_{\Sigma_m})=\mathrm{scal}(N)\geq 0$.
\end{enumerate}
\end{prop}
\begin{proof}
This literally follows from tracing the proof of each inequalities in the work of \cite[Section 3]{BrendleHirschJohne2022}. We point them out one by one. Combining Lemma \ref{lma:Stable-ineq}, \ref{lma:BHJ-ineq}, $\Lambda\geq 0$ and $\mathcal{C}_{m}\geq0$, we obtain on $\Sigma_{m}$,
\begin{equation}
\Lambda = 0, \ \ \mathcal{C}_{m}(e_{1},\ldots,e_{m}) = 0, \ \ \mathcal{V}_k = 0 \ \ \text{for all $1\leq k\leq m$}.
\end{equation}
Then we obtain (i) and (v).
The proof of \cite[Lemma 3.7]{BrendleHirschJohne2022} actually shows
\begin{equation}
\mathcal{G} \geq \sum_{k=2}^{m}\left(\frac{1}{2}+\frac{1}{2(k-1)}\right)H_{\Sigma_{k}}^{2}+\frac{m}{2(m-1)}|\nabla_{\Sigma_{m}}\log\rho_{m-1}|^{2}.
\end{equation}
The equality holds only if $\nabla_{\Sigma_m}\log \rho_{m-1}\equiv 0$ and hence $\rho_{m-1}$ is constant along $\Sigma_m$. By the stability inequality of $\Sigma_{m}$, for any $\psi\in C^{\infty}(\Sigma_{m})$, we have
\begin{equation*}
\begin{split}
0 \leq {} & -\int_{\Sigma_{m}}\psi\Delta_{\Sigma_{m}}\psi\, d\mu \\
& -\int_{\Sigma_{m}}\left(|h_{\Sigma_{m}}|^{2}
+\Ric_{\Sigma_{m-1}}(\nu_{m},\nu_{m})
-\nabla_{\Sigma_{m-1}}^{2}\log\rho_{m-1}(\nu_{m},\nu_{m})\right)\psi^{2}\, d\mu.
\end{split}
\end{equation*}
The equality in \cite[Lemma 3.3]{BrendleHirschJohne2022} holds for the test function $\psi=\rho_{m-1}^{-1}$ which is now constant, we see that on $\Sigma_m$,
\begin{equation}\label{m-eignfunction}
|h_{\Sigma_m}|^2+\Ric_{\Sigma_{m-1}}(\nu_m,\nu_m)-\left(\nabla^2_{\Sigma_{m-1}}\log\rho_{m-1}\right)(\nu_m,\nu_m) = 0.
\end{equation}
It remains to show (iv) and (vi). Indeed, if $h_{\Sigma_k}\equiv 0$ for all $1\leq k\leq m$ on $\Sigma_m$, then \eqref{m-eignfunction} implies (iii), and (ii) follows from $\nabla_{\Sigma_m}\log \rho_{m-1}\equiv 0$ and the first variation of $\Sigma_{m}$:
\begin{equation}
H_{\Sigma_{m}} + \langle\nabla_{\Sigma_{m-1}}\log\rho_{m-1},\nu_{m}\rangle = 0.
\end{equation}
In fact, (iv) has been proved implicitly in \cite{BrendleHirschJohne2022}. It suffices to trace the equality case in the argument of \cite[Lemma 3.11, 3.12 and 3.13]{BrendleHirschJohne2022}.   For the reader's convenience, we include a proof here.

To show that $h_{\Sigma_{1}}\equiv0$ on $\Sigma_{m}$, we compute
\begin{equation*}
\begin{split}
\mathcal{V}_{1} = {} & |h_{\Sigma_{1}}|^{2}+\sum_{p=2}^{m}\sum_{q=p+1}^{n}\left(h_{\Sigma_{1}}(e_{p},e_{p})h_{\Sigma_{1}}(e_{q},e_{q})-h_{\Sigma_{1}}(e_{p},e_{q})^{2}\right) \\
= {} & \sum_{p=2}^{n}h_{\Sigma_{1}}(e_{p},e_{p})^{2}+2\sum_{p=2}^{n-1}\sum_{q=p+1}^{n}h_{\Sigma_{1}}(e_{p},e_{q})^{2} \\
& +\sum_{p=2}^{m}\sum_{q=p+1}^{n}h_{\Sigma_{1}}(e_{p},e_{p})h_{\Sigma_{1}}(e_{q},e_{q})
-\sum_{p=2}^{m}\sum_{q=p+1}^{n}h_{\Sigma_{1}}(e_{p},e_{q})^{2} \\
\geq {} & \sum_{p=2}^{n}h_{\Sigma_{1}}(e_{p},e_{p})^{2}+\sum_{p=2}^{m}\sum_{q=p+1}^{n}h_{\Sigma_{1}}(e_{p},e_{p})h_{\Sigma_{1}}(e_{q},e_{q})
+\sum_{p=2}^{n-1}\sum_{q=p+1}^{n}h_{\Sigma_{1}}(e_{p},e_{q})^{2}.
\end{split}
\end{equation*}
For the second term on the right hand side,
\begin{equation}
\begin{split}
& \sum_{p=2}^{m}\sum_{q=p+1}^{n}h_{\Sigma_{1}}(e_{p},e_{p})h_{\Sigma_{1}}(e_{q},e_{q}) \\
= {} & \sum_{p=2}^{m}h_{\Sigma_{1}}(e_{p},e_{p})\left(H_{\Sigma_{1}}-\sum_{q=2}^{p} h_{\Sigma_{1}}(e_{q},e_{q})\right) \\
= {} & \sum_{p=2}^{m}h_{\Sigma_{1}}(e_{p},e_{p})H_{\Sigma_{1}}-\sum_{p=2}^{m}\sum_{q=2}^{p}h_{\Sigma_{1}}(e_{p},e_{p})h_{\Sigma_{1}}(e_{q},e_{q}) \\
= {} & \sum_{p=2}^{m}h_{\Sigma_{1}}(e_{p},e_{p})H_{\Sigma_{1}}-\frac{1}{2}\sum_{p=2}^{m}h_{\Sigma_{1}}(e_{p},e_{p})^{2}
-\frac{1}{2}\left(\sum_{p=2}^{m}h_{\Sigma_{1}}(e_{p},e_{p})\right)^{2}.
\end{split}
\end{equation}
Then we have
\begin{equation}
\begin{split}
\mathcal{V}_{1} \geq {} & \frac{1}{2}\sum_{p=2}^{m}h_{\Sigma_{1}}(e_{p},e_{p})^{2}+\sum_{q=m+1}^{n}h_{\Sigma_{1}}(e_{q},e_{q})^{2}
+\sum_{p=2}^{n-1}\sum_{q=p+1}^{n}h_{\Sigma_{1}}(e_{p},e_{q})^{2} \\
& +\sum_{p=2}^{m}h_{\Sigma_{1}}(e_{p},e_{p})H_{\Sigma_{1}}-\frac{1}{2}\left(\sum_{p=2}^{m}h_{\Sigma_{1}}(e_{p},e_{p})\right)^{2}.
\end{split}
\end{equation}
Combining this with $H_{\Sigma_{1}}\equiv0$ and Lemma \ref{V H C vanish} ($\mathcal{V}_{1}\equiv0$, $\sum_{p=2}^{m}h_{\Sigma_{1}}(e_{p},e_{p})\equiv0$), we obtain $h_{\Sigma_{1}}\equiv0$.

To prove $h_{\Sigma_{k}}\equiv0$ on $\Sigma_{m}$, we compute
\begin{equation*}
\begin{split}
\mathcal{V}_{k} = {} & |h_{\Sigma_{k}}|^{2}-\left(\frac{1}{2}-\frac{1}{2(k-1)}\right)H_{\Sigma_{k}}^{2} \\
& + \sum_{p=k+1}^{m}\sum_{q=p+1}^{n}\left(h_{\Sigma_{k}}(e_{p},e_{p})h_{\Sigma_{k}}(e_{q},e_{q})-h_{\Sigma_{k}}(e_{p},e_{q})^{2}\right) \\
= {} & \sum_{p=k+1}^{n}h_{\Sigma_{1}}(e_{p},e_{p})^{2}+2\sum_{p=k+1}^{n-1}\sum_{q=p+1}^{n}h_{\Sigma_{k}}(e_{p},e_{q})^{2}
-\left(\frac{1}{2}-\frac{1}{2(k-1)}\right)H_{\Sigma_{k}}^{2} \\
& + \sum_{p=k+1}^{m}\sum_{q=p+1}^{n}h_{\Sigma_{k}}(e_{p},e_{p})h_{\Sigma_{k}}(e_{q},e_{q})-\sum_{p=k+1}^{m}\sum_{q=p+1}^{n}h_{\Sigma_{k}}(e_{p},e_{q})^{2} \\
\geq {} & \sum_{p=k+1}^{n}h_{\Sigma_{1}}(e_{p},e_{p})^{2}+\sum_{p=k+1}^{n-1}\sum_{q=p+1}^{n}h_{\Sigma_{k}}(e_{p},e_{q})^{2} \\
& -\left(\frac{1}{2}-\frac{1}{2(k-1)}\right)H_{\Sigma_{k}}^{2}+\sum_{p=k+1}^{m}\sum_{q=p+1}^{n}h_{\Sigma_{k}}(e_{p},e_{p})h_{\Sigma_{k}}(e_{q},e_{q}).
\end{split}
\end{equation*}
For the third term on the right hand side,
\begin{equation}
\begin{split}
& -\left(\frac{1}{2}-\frac{1}{2(k-1)}\right)H_{\Sigma_{k}}^{2} \\
= {} & -\left(\frac{1}{2}-\frac{1}{2(k-1)}\right)\left(\sum_{p=k+1}^{m}h_{\Sigma_{k}}(e_{p},e_{p})+\sum_{q=m+1}^{n}h_{\Sigma_{k}}(e_{q},e_{q})\right)^{2} \\
= {} & -\left(\frac{1}{2}-\frac{1}{2(k-1)}\right)\left(\sum_{p=k+1}^{m}h_{\Sigma_{k}}(e_{p},e_{p})\right)^{2} \\
& -\left(\frac{1}{2}-\frac{1}{2(k-1)}\right)\left(\sum_{q=m+1}^{n}h_{\Sigma_{k}}(e_{q},e_{q})\right)^{2} \\
& -\left(1-\frac{1}{k-1}\right)\left(\sum_{p=k+1}^{m}h_{\Sigma_{k}}(e_{p},e_{p})\right)\left(\sum_{q=m+1}^{n}h_{\Sigma_{k}}(e_{q},e_{q})\right).
\end{split}
\end{equation}
Then we have
\begin{equation}
\begin{split}
\mathcal{V}_{k} \geq I_{1}+I_{2},
\end{split}
\end{equation}
where
\begin{equation*}
\begin{split}
I_{1} = {} & \sum_{p=k+1}^{n}h_{\Sigma_{1}}(e_{p},e_{p})^{2}+\sum_{p=k+1}^{n-1}\sum_{q=p+1}^{n}h_{\Sigma_{k}}(e_{p},e_{q})^{2} \\
& +\frac{1}{2(k-1)}\left(\sum_{p=k+1}^{m}h_{\Sigma_{k}}(e_{p},e_{p})\right)^{2}
-\left(\frac{1}{2}-\frac{1}{2(k-1)}\right)\left(\sum_{q=m+1}^{n}h_{\Sigma_{k}}(e_{q},e_{q})\right)^{2} \\
& +\frac{1}{k-1}\left(\sum_{p=k+1}^{m}h_{\Sigma_{k}}(e_{p},e_{p})\right)\left(\sum_{q=m+1}^{n}h_{\Sigma_{k}}(e_{q},e_{q})\right).
\end{split}
\end{equation*}
and
\begin{equation*}
\begin{split}
I_{2} = {} & -\frac{1}{2}\left(\sum_{p=k+1}^{m}h_{\Sigma_{k}}(e_{p},e_{p})\right)^{2}
-\left(\sum_{p=k+1}^{m}h_{\Sigma_{k}}(e_{p},e_{p})\right)\left(\sum_{q=m+1}^{n}h_{\Sigma_{k}}(e_{q},e_{q})\right) \\
& +\sum_{p=k+1}^{m}\sum_{q=p+1}^{n}h_{\Sigma_{k}}(e_{p},e_{p})h_{\Sigma_{k}}(e_{q},e_{q}).
\end{split}
\end{equation*}
Observe that
\begin{equation}
\begin{split}
I_{2} = {} & -\frac{1}{2}\left(\sum_{p=k+1}^{m}h_{\Sigma_{k}}(e_{p},e_{p})^{2}
+2\sum_{p=k+1}^{m-1}\sum_{q=p+1}^{m}h_{\Sigma_{k}}(e_{p},e_{p})h_{\Sigma_{k}}(e_{q},e_{q})\right) \\
& +\sum_{p=k+1}^{m}h_{\Sigma_{k}}(e_{p},e_{p})\left(\sum_{q=p+1}^{n}h_{\Sigma_{k}}(e_{q},e_{q})-\sum_{q=m+1}^{n}h_{\Sigma_{k}}(e_{q},e_{q})\right) \\
= {} & -\frac{1}{2}\sum_{p=k+1}^{m}h_{\Sigma_{k}}(e_{p},e_{p})^{2}.
\end{split}
\end{equation}
Then
\begin{equation*}
\begin{split}
\mathcal{V}_{k} \geq {} & \frac{1}{2}\sum_{p=k+1}^{m}h_{\Sigma_{1}}(e_{p},e_{p})^{2}+\sum_{q=m+1}^{n}h_{\Sigma_{1}}(e_{q},e_{q})^{2}
+\sum_{p=k+1}^{n-1}\sum_{q=p+1}^{n}h_{\Sigma_{k}}(e_{p},e_{q})^{2} \\
& +\frac{1}{2(k-1)}\left(\sum_{p=k+1}^{m}h_{\Sigma_{k}}(e_{p},e_{p})\right)^{2}
-\left(\frac{1}{2}-\frac{1}{2(k-1)}\right)\left(\sum_{q=m+1}^{n}h_{\Sigma_{k}}(e_{q},e_{q})\right)^{2} \\
& +\frac{1}{k-1}\left(\sum_{p=k+1}^{m}h_{\Sigma_{k}}(e_{p},e_{p})\right)\left(\sum_{q=m+1}^{n}h_{\Sigma_{k}}(e_{q},e_{q})\right).
\end{split}
\end{equation*}
Combining this with Lemma \ref{V H C vanish} ($\mathcal{V}_{k}\equiv0$, $\sum_{q=m+1}^{n}h_{\Sigma_{k}}(e_{q},e_{q})\equiv0$), we obtain $h_{\Sigma_{k}}\equiv0$ for $2\leq k\leq m-1$.

When $k=m$, it is clear that $h_{\Sigma_{m}}\equiv0$ follows directly from the definition of $\mathcal{V}_{m}$ and Lemma \ref{V H C vanish} ($\mathcal{V}_{m}\equiv0$, $H_{\Sigma_{m}}\equiv0$). Then we obtain (iv).

For (vi), (iv) shows $(\Sigma_{m},g|_{\Sigma_{m}})\to(N,g)$ is totally geodesic and then
\[
\begin{split}
\mathrm{scal}(g|_{\Sigma_m}) = {} & 2\sum_{p=m+1}^{n-1}\sum_{q=p+1}^{n}\mathrm{Rm}_{\Sigma_{m}}(e_{p},e_{q},e_{p},e_{q}) \\
= {} & 2\sum_{p=m+1}^{n-1}\sum_{q=p+1}^{n}\mathrm{Rm}_{N}(e_{p},e_{q},e_{p},e_{q}) \\
= {} & 2\sum_{p=1}^{n-1}\sum_{q=p+1}^{n}\mathrm{Rm}_{N}(e_{p},e_{q},e_{p},e_{q})-2\sum_{p=1}^{m}\sum_{q=p+1}^{n}\mathrm{Rm}_{N}(e_{p},e_{q},e_{p},e_{q}) \\[3mm]
= {} & \mathrm{scal}(N)-2\mathcal{C}_{m}(e_{1},\ldots,e_{m}).
\end{split}
\]
Since $(N,g)$ has non-negative $m$-intermediate curvature, then $\mathrm{scal}(N)\geq0$. Combining this with $\mathcal{C}_{m}(e_{1},\ldots,e_{m})=0$, we obtain (vi).
\end{proof}


\section{Foliation of $\Sigma_m$ under $\mathcal{C}_m\geq 0$}

In this section, we will use the idea of Zhu \cite{Zhu2020} to show that if $\mathcal{C}_{m}(e_{1},\ldots,e_{m}) \geq 0$, then the submanifold $\Sigma_m$ in Definition~\ref{defn:Stable weighted slicing of order m} admits a local foliation $\{\Sigma_{m,t}\}_{-\e\leq t\leq \e}$ so that each $\Sigma_{m,t}$ is also a minimizer of the weighed area.  We first find the local foliation using the argument in \cite[Lemma 3.3]{Zhu2020}.

\begin{lma}\label{lma:local-foliation}
If $n(m-2)<m^{2}-2$, then we can find a local foliation $\{\Sigma_{m,t}\}_{-\e< t< \e}$ of $\Sigma_m$ in $\Sigma_{m-1}$ such that on each $\Sigma_{m,t}$ is given by the graph over $\Sigma_m$ with graph function $u_t$ along the unit normal $\nu_m$ such that
\begin{equation}
\Sigma_{m,0}=\Sigma_m, \quad \frac{\partial}{\partial t} \bigg|_{t=0} u_t=1, \quad
\frac{\p u_{t}}{\p t} > 0, \quad
\fint_{\Sigma_m}u_t\,d\mu =t
\end{equation}
and $H_{\Sigma_{m,t}}+\langle\nabla_{\Sigma_{m-1}}\log \rho_{m-1},\nu_{m,t}\rangle$ is constant on $\Sigma_{m,t}$, where $\nu_{m,t}$ is unit normal of $\Sigma_{m,t}$.
\end{lma}
\begin{proof}
This follows by modifying the proof of \cite[Lemma 3.3]{Zhu2020} slightly.  We include a sketch for the reader's convenience. For $u\in C^{2,\a}(\Sigma_m)$, let $\Sigma_{u}$ be the graph over $\Sigma_m$ with graph function $u_t$ along the unit normal $\nu_m$, and $\nu_{u}$ be the unit normal of $\Sigma_{u}$. Consider the map $\Psi: C^{2,\a}(\Sigma_m)\to \mathring{C}^\a(\Sigma_m)\times \mathbb{R}$ given by
\begin{equation}
\Psi(u)=\left(\tilde H_u-\fint_{\Sigma_m}\tilde H_u \,d\mu, \ \fint_{\Sigma_m}u \,d\mu  \right)
\end{equation}
where $\mathring{C}^\a(\Sigma_m)$ is the space of function $f$ in $C^\a(\Sigma_m)$ such that $\int_{\Sigma_m}f\,d\mu=0$ and $\tilde H_{\Sigma_{u}}=H_{\Sigma_{u}}+\langle\nabla_{\Sigma_{m-1}}\log \rho_{m-1},\nu_{u}\rangle$ is the sum of mean curvature and a twist term given by the inner product.

By the second variation formula (for instances see \cite[Proposition 2.3]{BrendleHirschJohne2022}) and \eqref{m = 1 case} when $m=1$ or (iii), (iv) in Proposition~\ref{Prop:rigidity} when $m\geq2$, it is clear that $\Psi(0)=(0,0)$ and the linearised operator of $\Psi$ at $u=0$ is given by
\begin{equation}
D\Psi\big|_{u=0}(v)=\left(-\Delta_{\Sigma_m} v, \, \fint_{\Sigma_m}v \,d\mu\right)\in  \mathring{C}^\a(\Sigma_m)\times \mathbb{R}
\end{equation}
is invertible. By inverse function theorem, we can find a family $u_t: \Sigma_m\to \mathbb{R}$ for $t\in (-\e,\e)$ such that $\Psi(u_{t})=(0,t)$ and $u_{0}=0$. Denote $\Sigma_{u_{t}}$ by $\Sigma_{m,t}$. It follows that
\begin{equation}
\Sigma_{m,0}=\Sigma_m, \quad \frac{\partial}{\partial t} \bigg|_{t=0} u_t=1, \quad \fint_{\Sigma_m}u_t\,d\mu =t
\end{equation}
and $H_{\Sigma_{m,t}}+\langle\nabla_{\Sigma_{m-1}}\log \rho_{m-1},\nu_{m,t}\rangle$ is constant on $\Sigma_{m,t}$. By shrinking $\e$, we may assume $\partial_t u_t>0$. This completes the proof.
\end{proof}

Next, we want to show that $\Sigma_{m,t}$ obtained from Lemma~\ref{lma:local-foliation} are all local minimizer. To establish this, we need to construct competitors by modifying \cite[Section 4]{BrendleHirschJohne2022}. From now on, we assume working in the setting of \cite[Theorem 1.5]{BrendleHirschJohne2022}.  More precisely, we consider a compact manifold $N^n$ such that it admits a metric with $\mathcal{C}_m\geq 0$ for some $1\leq m\leq n-1$ and admits a non-zero degree map $f:N^n\to M^{n-m}\times \mathbb{T}^m$ for some compact manifold $M^{n-m}$. In \cite[Theorem 1.5]{BrendleHirschJohne2022},  it is proved that the stable weighed slicing in Definition~\ref{defn:Stable weighted slicing of order m} exists. Let $\Sigma_m$ be a submanifold obtained from this and $\{\Sigma_{m,t}\}_{-\e<t<\e}$ be the corresponding foliation obtained using Lemma~\ref{lma:local-foliation}.
\begin{prop}\label{prop:minimial-folation}
If $n(m-2)<m^{2}-2$ and $n(m-1)<m(m+1)$, then for each $t\in (-\e,\e)$,  we have
\begin{equation}\label{constant-weighted-area}
\mathcal{H}^{n-m}_{\rho_{m-1}}(\Sigma_{m,t})=\mathcal{H}^{n-m}_{\rho_{m-1}}(\Sigma_m)=\int_{\Sigma_m} \rho_{m-1} \, d\mu.
\end{equation}
In particular, each $\Sigma_{m,t}$ satisfies the conclusions in Proposition~\ref{Prop:rigidity}.
\end{prop}
\begin{proof}
The argument is based on modifying that of \cite[Proposition 3.4]{Zhu2020} using the idea of \cite{BrendleHirschJohne2022}.

We first recall some set-up in the proof of \cite[Theorem 1.5]{BrendleHirschJohne2022}. We may assume $\mathrm{deg}(f)>0$. Since $f:N\to M\times \mathbb{T}^m$ is a map with non-zero degree, we obtain the following maps
\begin{equation}
f_{0}: N\to M, \ \ f_i: N\to \mathbb{S}^1 \ \ i = 1,\ldots,m,
\end{equation}
where $f_0$ and $f_{i}$ correspond the projection from $M\times \mathbb{T}^m$ to $M$ and $i$-th torus $\mathbb{T}$ respectively. Then we choose a top-dimensional form $\Theta$, $\theta$ in $M$, $\mathbb{S}^1$ respectively, and define $\Omega=f_0^*\Theta$, $\omega_i=f_i^*\theta$ on $N$ so that $\int_{N} \omega_1\wedge ...\wedge \omega_m \wedge \Omega=\mathrm{deg}(f)$.

Next, we will show \eqref{constant-weighted-area} for $t\in(0,\e)$. When $t\in(-\e,0)$, the argument is similar. Recall that $\rho_0=1$, $\Sigma_0=N$ and the $\Sigma_k$ is obtained inductively by minimizing the weighed area inside the class $\mathcal{A}_k$ which contains all $(n-k)$ integer rectifiable current $\Sigma$ in $\Sigma_{k-1}$ with $\int_\Sigma \omega_{k+1}\wedge ...\wedge \omega_m \wedge \Omega=\mathrm{deg}(f)$. By Stokes Theorem, it is clear that $\Sigma_{m,t}\in \mathcal{A}_m$ and hence,
\begin{equation}
\begin{split}
0&\leq  \int_{\Sigma_{m,t}}\rho_{m-1} \;d\mu - \int_{\Sigma_{m,0}}\rho_{m-1} \;d\mu \\
 &\leq \int^t_0 \int_{\Sigma_{m,s}} \rho_{m-1} u_s \left( H_{m,s}  +\langle\nabla_{\Sigma_{m-1}}\log \rho_{m-1},\nu \rangle\right)d\mu \, ds.
\end{split}
\end{equation}
Therefore, to show \eqref{constant-weighted-area}, it suffices to show that
\begin{equation}\label{prop:minimial-folation goal}
H_{\Sigma_{m,s}}+\langle\nabla_{\Sigma_{m-1}}\log \rho_{m-1},\nu_{m,s} \rangle\leq 0 \ \ \text{for $s\in(0,\e)$}.
\end{equation}
Suppose on the contrary, there exists $\tau\in (0,\e)$ and $\delta>0$ such that
\begin{equation}
H_{\Sigma_{m,\tau}}+\langle\nabla_{\Sigma_{m-1}}\log \rho_{m-1},\nu_{m,\tau} \rangle > 2\delta > 0.
\end{equation}
We consider the following minimization problem. Define the twisted brane functional
\begin{equation}
\mathcal{B}(\hat \Omega)=\int_{\partial \hat \Omega\setminus \Sigma_m}\rho_{m-1} \,d\mu -\delta \int_{\hat \Omega}\rho_{m-1} \,dV
\end{equation}
for any Borel subset $\hat \Omega$ of the region between $\Sigma_m$ and $\Sigma_{m,\tau}$, with finite perimeter and $\Sigma_m\subset \partial\hat\Omega$. Due to the choice of $\delta$,  the hypersurface $\Sigma_m$ and $\Sigma_{m,\tau}$ serve as the barriers. Therefore, we can find a Borel set $\hat \Omega$ which is a local minimizer of $\mathcal{B}$ which $\hat\Sigma_m=\partial\hat\Omega\setminus \Sigma_m$ is a smooth $2$-sided hypersurface disjoint from $\Sigma_m$ and $\Sigma_{m,\tau}$.  Moreover, the first and second variation of $\hat \Omega$ implies that
\begin{equation}\label{1st and 2nd variation}
\left\{
\begin{array}{ll}
H_{\hat \Sigma_m}+\langle\nabla_{\Sigma_{m-1}}\log \rho_{m-1},\hat\nu_m \rangle=\delta \ \ \text{on $\hat\Sigma_m$};\\[2mm]
\displaystyle \int_{\hat \Sigma_m} \rho_{m-1}\cdot f \cdot L(f)\, d\mu \geq 0 \ \ \text{for any $f\in C^\infty(\hat \Sigma)$},
\end{array}
\right.
\end{equation}
where $L(f)$ is the stability operator given by
\begin{equation}
\begin{split}
L(f)&=-\Delta_{\hat\Sigma_m}f- \left(|h_{\hat{\Sigma}_{m}}|^2+\Ric_{\Sigma_{m-1}}(\hat\nu_m,\hat \nu_m) \right)f\\[2mm]
&\quad + \left( \nabla^2_{\Sigma_{m-1}}\log\rho_{m-1}(\hat\nu_m,\hat\nu_m) \right)f-\langle \nabla_{\hat\Sigma_m}\log\rho_{m-1},\nabla_{\hat\Sigma_m} f\rangle.
\end{split}
\end{equation}
Compared to $\Sigma_{m}$, we observe that $\hat\Sigma_m$ satisfies the same second variation while the first variation involve an extra $\delta>0$.  This allow us to carry out similar analysis as in \cite{BrendleHirschJohne2022} with the modification accommodating the twisting from the mean curvature on $\hat\Sigma_m$. Since $\hat\Sigma_m$ is the hypersurface of $\Sigma_{m-1}$, it suffices to modify the argument of \cite[Section 3]{BrendleHirschJohne2022} which uses $H_{\Sigma_m}+\langle\nabla_{\Sigma_{m-1}}\log \rho_{m-1},\nu_m\rangle=0$ on $\Sigma_m$.

We start with the stability inequality on $\hat \Sigma_m$, i.e. the analogous inequality in \cite[Lemma 3.3]{BrendleHirschJohne2022}. Taking $f=\rho_{m-1}^{-1}$ in \eqref{1st and 2nd variation}, we obtain
\begin{equation}
\begin{split}
& \int_{\hat{\Sigma}_{m}}\rho_{m-1}^{-1}\left(\Delta_{\hat{\Sigma}_{m}}\log\rho_{m-1}+(\nabla_{\hat{\Sigma}_{m-1}}^{2}\log\rho_{m-1})(\hat{\nu}_{m},\hat{\nu}_{m})\right)d\mu \\
\geq {} & \int_{\hat{\Sigma}_{m}}\rho_{m-1}^{-1}\left(|h_{\hat{\Sigma}_{m}}|^{2}+\Ric_{\Sigma_{m-1}}(\hat{\nu}_{m},\hat{\nu}_{m})\right)d\mu.
\end{split}
\end{equation}
Note that
\begin{equation}
\begin{split}
& \Delta_{\hat{\Sigma}_{m}}\log\rho_{m-1}+(\nabla_{\hat{\Sigma}_{m-1}}^{2}\log\rho_{m-1})(\hat{\nu}_{m},\hat{\nu}_{m}) \\[1mm]
= {} & \Delta_{\Sigma_{m-1}}\log\rho_{m-1}-H_{\hat{\Sigma}_{m}}\langle \nabla_{\Sigma_{m-1}}\log\rho_{m-1},\hat{\nu}_{m} \rangle \\[1mm]
= {} & \Delta_{\Sigma_{m-1}}\log\rho_{m-1}+H_{\hat{\Sigma}_{m}}^{2}-\delta H_{\hat{\Sigma}_{m}}.
\end{split}
\end{equation}
It then follows that
\begin{equation}\label{new-stability 1}
\begin{split}
& \int_{\hat{\Sigma}_{m}}\rho_{m-1}^{-1}\left(\Delta_{\Sigma_{m-1}}\log\rho_{m-1}+H_{\hat{\Sigma}_{m}}^{2}-\delta H_{\hat{\Sigma}_{m}}\right)d\mu \\
\geq {} & \int_{\hat{\Sigma}_{m}}\rho_{m-1}^{-1}\left(|h_{\hat{\Sigma}_{m}}|^{2}+\Ric_{\Sigma_{m-1}}(\hat{\nu}_{m},\hat{\nu}_{m})\right)d\mu.
\end{split}
\end{equation}
If $m=1$, then $\mathcal{C}_{1}\geq0$ is equivalent to $\Ric\geq0$. By $\rho_{0}\equiv1$ and \eqref{1st and 2nd variation}, we obtain $H_{\hat{\Sigma}_{1}}=\delta$. Then \eqref{new-stability 1} implies
\begin{equation}
0 \geq \int_{\hat{\Sigma}_{1}}|h_{\hat{\Sigma}_{1}}|^{2}\,d\mu \geq \frac{1}{n-1}\int_{\hat{\Sigma}_{1}}H_{\hat{\Sigma}_{1}}^{2}\,d\mu 
= \frac{\delta^{2}}{n-1}\int_{\hat{\Sigma}_{1}}d\mu > 0,
\end{equation}
which is a contradiction. This shows \eqref{prop:minimial-folation goal} holds. If $m\geq2$, for $1\leq k\leq m-1$, \cite[Lemma 3.1 and 3.2]{BrendleHirschJohne2022} imply
\begin{equation*}
\begin{split}
\Delta_{\Sigma_{k}}\log\rho_{k} = {} & \Delta_{\Sigma_{k}}\log\rho_{k-1}+(\nabla_{\Sigma_{k-1}}^{2}\log\rho_{k-1})(\nu_{k},\nu_{k}) \\[1mm]
& -\left(\lambda_{k}+|h_{\Sigma_{k}}|^{2}+\Ric_{\Sigma_{k-1}}(\nu_{k},\nu_{k})+\langle \nabla_{\Sigma_{k}}\log\rho_{k},\nabla_{\Sigma_{k}}w_{k}\rangle\right) \\[1mm]
= {} & \Delta_{\Sigma_{k-1}}\log\rho_{k-1}+H_{\Sigma_{k}}^{2} \\[1mm]
& -\left(\lambda_{k}+|h_{\Sigma_{k}}|^{2}+\Ric_{\Sigma_{k-1}}(\nu_{k},\nu_{k})+\langle \nabla_{\Sigma_{k}}\log\rho_{k},\nabla_{\Sigma_{k}}w_{k}\rangle\right).
\end{split}
\end{equation*}
Combining this with $\rho_{0}\equiv1$,
\begin{equation*}
\Delta_{\Sigma_{m-1}}\log\rho_{m-1}
= \sum_{k=1}^{m-1}H_{\Sigma_{k}}^{2} -\sum_{k=1}^{m-1}\left(\lambda_{k}+|h_{\Sigma_{k}}|^{2}+\Ric_{\Sigma_{k-1}}(\nu_{k},\nu_{k})+\langle \nabla_{\Sigma_{k}}\log\rho_{k},\nabla_{\Sigma_{k}}w_{k}\rangle\right).
\end{equation*}
Substituting this into \eqref{new-stability 1}, we obtain the following integral inequality:
\begin{equation}\label{new-stability}
\int_{\hat{\Sigma}_{m}}\rho_{m-1}^{-1}(\hat{\Lambda}+\hat{\mathcal{R}}+\hat{\mathcal{G}}+\hat{\mathcal{E}})\,d\mu \leq 0,
\end{equation}
where
\begin{equation}\label{hat Lambda R G E}
\left\{
\begin{array}{ll}
\displaystyle \hat{\Lambda} = \sum_{k=1}^{m-1}\lambda_{k};\\[5mm]
\displaystyle \hat{\mathcal{R}} = \sum_{k=1}^{m-1}\Ric_{\Sigma_{k-1}}(\nu_{k},\nu_{k})+\Ric_{\Sigma_{m-1}}(\hat\nu_{m},\hat{\nu}_{m});\\[5mm]
\displaystyle \hat{\mathcal{G}} = \sum_{k=1}^{m-1}\langle \nabla_{\Sigma_{k}}\log\rho_{k},\nabla_{\Sigma_{k}}w_{k}\rangle, \\[5mm]
\displaystyle \hat{\mathcal{E}} = \sum_{k=1}^{m-1}|h_{\Sigma_{k}}|^{2}+|h_{\hat{\Sigma}_{m}}|^{2}
-\sum_{k=2}^{m-1}H_{\Sigma_{k}}^{2}-H_{\hat{\Sigma}_{m}}^{2}+\delta H_{\hat{\Sigma}_{m}}.
\end{array}
\right.
\end{equation}

\bigskip
\noindent
{\bf Claim.}
For $x\in\hat{\Sigma}_{m}$, let $\{e_{i}\}_{i=1}^{n}$ be an orthonormal basis of $T_{x}N$ such that $e_{k}=\nu_{k}$ for $1\leq k\leq m-1$ and $e_{m}=\hat{\nu}_{m}$. Then the following holds:
\begin{equation}\label{claim}
\hat{\mathcal{R}}+\hat{\mathcal{G}}+\hat{\mathcal{E}} \geq \mathcal{C}_{m}(e_{1},\ldots,e_{m})+\sum_{k=1}^{m}\hat{\mathcal{V}}_{k},
\end{equation}
where $\hat{\mathcal{V}}_{k}=\mathcal{V}_{k}$ in Lemma \ref{lma:BHJ-ineq} for $1\leq k\leq m-1$ and
\begin{equation}
\hat{\mathcal{V}}_{m} = |h_{\Sigma_{m}}|^{2}-H_{\hat{\Sigma}_{m}}^{2}+\delta H_{\hat{\Sigma}_{m}}+\frac{m}{2(m-1)}(H_{\hat{\Sigma}_{m}}-\delta)^{2}.
\end{equation}

\begin{proof}[Proof of Claim]
The argument of \cite[Lemma 3.8]{BrendleHirschJohne2022} shows
\begin{equation*}
\hat{\mathcal{R}} = \mathcal{C}_{m}(e_{1},\ldots,e_{m})
+\sum_{k=1}^{m-1}\sum_{p=k+1}^{m}\sum_{q=p+1}^{n}
\left(h_{\Sigma_{k}}(e_{p},e_{p})h_{\Sigma_{k}}(e_{q},e_{q})-h_{\Sigma_{k}}(e_{p},e_{q})^{2}\right).
\end{equation*}
For the lower bound of $\hat{\mathcal{G}}$, replacing $H_{\Sigma_m}+\langle\nabla_{\Sigma_{m-1}}\log \rho_{m-1},\nu_m\rangle=0$ by $H_{\hat \Sigma_m}+\langle\nabla_{\Sigma_{m-1}}\log \rho_{m-1},\hat\nu_m \rangle=\delta$ in the argument of \cite[Lemma 3.7]{BrendleHirschJohne2022}, we see that
\begin{equation}
\begin{split}
\hat{\mathcal{G}}&\geq \sum_{k=2}^{m-1} \left(\frac12+\frac1{2(k-1)} \right)H_{\Sigma_k}^2+\frac{m}{2(m-1)}\left|\langle  \nabla_{\Sigma_{m-1}}\log\rho_{m-1},\hat\nu_m\rangle\right|^{2}\\
&\geq \sum_{k=2}^{m-1} \left(\frac12+\frac1{2(k-1)} \right)H_{\Sigma_k}^2+\frac{m}{2(m-1)}(H_{\hat \Sigma_m}-\delta)^2.\\
\end{split}
\end{equation}
Then \eqref{claim} follows from the above and \eqref{hat Lambda R G E}.
\end{proof}

Substituting \eqref{claim} into \eqref{new-stability}, and using $\mathcal{C}_{m}\geq0$, $\hat{\Lambda}\geq0$ and Lemma \ref{lma:BHJ-ineq} ($\hat{\mathcal{V}}_{k}=\mathcal{V}_{k}\geq0$ for $1\leq k\leq m-1$), we obtain
\begin{equation}
\begin{split}
0&\geq \int_{\hat \Sigma_m} \rho_{m-1}^{-1} \hat {\mathcal{V}}_m\, d\mu.
\end{split}
\end{equation}
However, direct calculation shows
\begin{equation}
\begin{split}
\hat{\mathcal{V}}_{m}
\geq {} & \frac{H_{\hat{\Sigma}_{m}}^{2}}{n-m}-H_{\hat{\Sigma}_{m}}^{2}+\delta H_{\hat{\Sigma}_{m}}+\frac{m}{2(m-1)}(H_{\hat{\Sigma}_{m}}-\delta)^{2}\\
= {} & \frac{m^2-2-n(m-2)}{2(n-m)(m-1)}H_{\hat \Sigma_m}^2 - \frac{\delta}{m-1} H_{\hat \Sigma_m} +\frac{m\delta^2}{2(m-1)}\\
\geq {} & -\frac{(n-m)\delta^2}{2(m-1)(m^2-2-n(m-2))}+\frac{m\delta^2}{2(m-1)}\\[2mm]
> {} & 0,
\end{split}
\end{equation}
provided that $n(m-1)<m(m+1)$, which is impossible. This contradiction shows that \eqref{prop:minimial-folation goal} holds, which completes the proof.
\end{proof}

\section{Proofs of Theorem~\ref{main-rigidity}, Corollary \ref{main-rigidity-quasi} and \ref{main-topo-lowD}}

After establishing the weighed foliation in section 3,  we are now ready to prove Theorem~\ref{main-rigidity}.
\begin{proof}[Proof of Theorem~\ref{main-rigidity}]
By \cite[Theorem 1.5]{BrendleHirschJohne2022}, the stable weighted slicing in Definition~\ref{defn:Stable weighted slicing of order m} exists.  Moreover, Proposition~\ref{prop:minimial-folation} infers that there exists a local foliation $\{\Sigma_{m,t}\}_{-\e<t<\e}$ of $\Sigma_m$ such that each $\Sigma_{m,t}$ is still a stable minimizer of the weighed area.

Let $g_{m-1}$ and $g_{m,t}$ be the induced metric of $g$ on $\Sigma_{m-1}$ and $\Sigma_{m,t}$ respectively.  We can write $g_{m-1}=f^2_t dt^2+ g_{m,t}$ locally where $f_t$ is the lapse function of $\Sigma_{m,t}$ moving along the foliation and $g_{m,0}=g_m$ is the induced metric on $\Sigma_m$.  Since for $t$ sufficiently small, $\Sigma_{m,t}$  is also a stable minimizer of the weighed area functional. By (ii) in Proposition~\ref{Prop:rigidity}, $\nabla_{\Sigma_{m-1}}\log\rho_{m-1}=0$ along the local foliation $\{\Sigma_{m,t}\}_{-\e<t<\e}$. Combining this with (iii) in Proposition~\ref{Prop:rigidity},
\begin{equation}
\Ric_{\Sigma_{m-1}}(\nu_{m,t},\nu_{m,t}) = \nabla^2_{\Sigma_{m-1}}\log\rho_{m-1}=0 \ \ \text{on $\Sigma_{m,t}$}.
\end{equation}
Using \eqref{constant-weighted-area}, we have
\begin{equation}
\frac{\p^{2}}{\p t^{2}}\int_{\Sigma_{m,t}} \rho_{m-1} \, d\mu = \frac{\p^{2}}{\p t^{2}}\mathcal{H}_{\rho_{m-1}}^{n-m}(\Sigma_{m,t}) = 0
\end{equation}
Combining the above, (iv) in Proposition \ref{Prop:rigidity} and second variation formula (for instances see \cite[Proposition 2.3]{BrendleHirschJohne2022}), we obtain $\Delta_{\Sigma_{m,t}} f_t=0$ and hence $f_t(x)=\phi(t)$ is a function depending only on $t$.  Lemma~\ref{lma:local-foliation} implies that $\phi(0)=1$. Using (iv) in Proposition \ref{Prop:rigidity} again, $\Sigma_{m,t}$ is totally geodesic and then $\partial_t g_{m,t}=0$, which implies $g_{m-1}=\phi(t)^2dt^2+g_m$ locally.  We now re-parametrize $t$ by letting
\begin{equation}
s(t)=\int^t_0 \phi(\tau)\,d\tau
\end{equation}
so that $g_{m-1}=ds^2+g_m$ locally around $\Sigma_m$. Applying the continuity argument as in \cite[Proposition 11]{BrayBrendleNeves2010}, we conclude that $\Sigma_{m-1}$ is isometrically covered by $\Sigma_m\times \mathbb{R}$ and hence $\rho_{m-1}$ is constant on $\Sigma_{m-1}$.

For $y\in\Sigma_{m-1}$, there is $(x,t)\in\Sigma_{m}\times\mathbb{R}$ which belongs to the preimage of $y$. We choose the orthonormal basis $\{e_i\}_{i=1}^n$ of $T_{y}N$ such that $e_{k}=\nu_{k}$ for $1\leq k\leq m-1$, $e_{m}$ is the unit vector of $\mathbb{R}$ and $\{e_{q}\}_{q=m+1}^{n}$ is an orthonormal basis of $T_{x}\Sigma_{m}$. It follows that
\begin{equation}
\mathrm{Rm}_{\Sigma_{m-1}}(e_m,e_q,e_m,e_q) = 0 \ \ \text{for $m+1\leq q\leq n$}.
\end{equation}
By (iv) and (v) in Proposition~\ref{Prop:rigidity}, we compute
\begin{equation*}
\begin{split}
\mathcal{C}_{m-1}(e_{1},\ldots,e_{m-1})&=\sum_{p=1}^{m-1} \sum_{q=p+1}^n \mathrm{Rm}_{N}(e_p,e_q,e_p,e_q)\\
&=\sum_{p=1}^{m} \sum_{q=p+1}^n \mathrm{Rm}_{N}(e_p,e_q,e_p,e_q)-\sum_{q=m+1}^n \mathrm{Rm}_{N}(e_m,e_q,e_m,e_q)\\
&=\mathcal{C}_{m}(e_{1},\ldots,e_{m})-\sum_{q=m+1}^n \mathrm{Rm}_{N}(e_m,e_q,e_m,e_q)\\
&=-\sum_{q=m+1}^n \mathrm{Rm}_{\Sigma_{m-1}}(e_m,e_q,e_m,e_q)=0.
\end{split}
\end{equation*}
To summarize, we have actually shown that if $\mathcal{C}_{m}(e_{1},\ldots,e_{m})=0$ on $\Sigma_{m}$, then $\mathcal{C}_{m-1}(e_{1},\ldots,e_{m-1})=0$ on $\Sigma_{m-1}$ and $\Sigma_{m-1}$ is isometrically covered by $\Sigma_m\times \mathbb{R}$. Repeating the argument inductively,  we see that  $\mathcal{C}_{k}(e_{1},\ldots,e_{k})=0$ and $\Sigma_{k-1}$ is isometrically covered by $\Sigma_{k}\times \mathbb{R}$ for all $1\leq k \leq m$. Therefore, $N=\Sigma_0$ is isometrically covered by $\Sigma_m^{n-m}\times \mathbb{R}^m$.

It remains to show that $\Ric(\Sigma_m)\geq 0$. Let $v\in T_x \Sigma_m$ with $|v|=1$. We choose an orthonormal basis $\{e_i\}_{i=1}^{n}$ of $(x,0^m)\in \Sigma_m^{n-m}\times \mathbb{R}^m$ such that $e_1=v$, $e_{k+1}$ is the unit vector on the $k$-th copy of $\mathbb{R}$ in $\mathbb{R}^m$ for $1\leq k\leq m$, and $\{e_{1}\}\cup\{e_{j}\}_{j=m+2}^{n-m}$ forms an orthonormal basis of $T_{x}\Sigma_{m}^{n-m}$. Using $\mathcal{C}_m\geq 0$ and (iv) in Proposition \ref{Prop:rigidity}, we have
\begin{equation}\label{C_mTORic}
\begin{split}
0\leq \mathcal{C}_m(e_{1},\ldots,e_{m})&=\sum_{p=1}^{m} \sum_{q=p+1}^n \mathrm{Rm}_{N}(e_p,e_q,e_p,e_q)\\
&= \sum_{q=2}^n \mathrm{Rm}_{N}(e_1,e_q,e_1,e_q)\\
&= \sum_{q=m+2}^n \mathrm{Rm}_{\Sigma_{m}}(e_1,e_q,e_1,e_q)\\[2mm]
&=\Ric_{\Sigma_m}(e_1,e_1).
\end{split}
\end{equation}

Finally, if $\mathrm{scal}(N)=0$, then $\mathcal{C}_m=0$ as $\mathrm{scal}(N)$ is the sum of $m$-intermediate curvature and $\mathcal{C}_m\geq 0$. The Ricci flatness follows from \eqref{C_mTORic}.
\end{proof}

With the local product metric structure, the quasi-positive case can be ruled out easily.
\begin{proof}[Proof of Corollary~\ref{main-rigidity-quasi}]
If there exists $x\in N$ such that $\mathcal{C}_m(e_1,...,e_m)>0$ for all possible $\{e_i\}_{i=1}^n$. By Theorem~\ref{main-rigidity}, we have  contradiction if we choose $e_k$ to be the unit vector on the $k$-th $\mathbb{R}$ factor for $1\leq k\leq m$.
\end{proof}

If $n-m\leq 3$, we can appeal to the classification of two-fold and three-fold to prove Corollary~\ref{main-topo-lowD}.

\begin{proof}[Proof of Corollary~\ref{main-topo-lowD}]
Since $N$ is connected, we may assume $X=\Sigma_m$ to be connected by choosing the component so that $N$ is isometrically covered by $\Sigma_m\times \mathbb{R}^{m}$ by Theorem~\ref{main-rigidity}. We note here that the non-zero degree of the mapping $f$ is only used to find the stable weighted slicing.

If $n-m=2$, then $X^{2}$ is a two-fold with $\mathrm{scal}(X)\geq 0$ and hence the it is homeomorphic to either $\mathbb{T}^2$ or $\mathbb{S}^2$ by Gauss-Bonnet Theorem.
Since $N$ is locally isometric to $X^2\times \mathbb{R}^{n-2}$, if $\mathrm{scal}(N)$ is positive at one point, then $\mathrm{scal}(X)>0$ somewhere and hence $X$ is homeomorphic to $\mathbb{S}^2$. Otherwise, $\mathrm{scal}(X)=\mathrm{scal}(N)=0$ and hence $X$ is a flat torus.

If $n-m=3$, then $X^3$ is a three-fold with $\Ric\geq 0$. The universal cover is either $\mathbb{S}^2\times \mathbb{R}$, $\mathbb{R}^3$ or $\mathbb{S}^3$ by the works of Hamilton \cite[Theorem 1.2]{Hamilton4}. In fact, it is determined by its Ricci curvature by Ricci flows and Hamilton's strong maximum principle. Since $N$ is isometrically covered by $X^3\times \mathbb{R}^m$, the curvature of $X$ is completely determined by that of $N$. If $\mathrm{scal}(N)\equiv 0$ or equivalently $\Ric(X^3)\equiv 0$ (since $\Ric\geq 0$), then it is isometrically covered by $\mathbb{R}^3$. If $\Ric(N)$ is strictly positive at some point, then so does $X$ and hence it is covered by $\mathbb{S}^3$. Otherwise, it is covered by $\mathbb{S}^2\times \mathbb{R}$.
\end{proof}

\end{document}